\documentclass[11pt,a4paper]{amsart}
\usepackage[top=1in, bottom=.8in, left=1in, right=1in]{geometry}
\usepackage{amsmath,enumerate}
\usepackage{amssymb}
\usepackage{amsfonts}
\usepackage{palatino, mathrsfs}
\usepackage{url}
\usepackage[utf8]{inputenc}
\usepackage[T1]{fontenc}
\usepackage{verbatim}
\usepackage[final]{prelim2e}
\usepackage{mathtools}
\usepackage{thmtools}
\usepackage{thm-restate}

\newtheorem{theorem}{Theorem}[section]
\theoremstyle{plain} 
\theoremstyle{plain} 
\newcommand{\thistheoremname}{}
\newtheorem{genericthm}[theorem]{\thistheoremname}

\newtheorem*{genericthm*}{\thistheoremname}
\newenvironment{namedthm*}[1]
{\renewcommand{\thistheoremname}{#1}%
	\begin{genericthm*}}
	{\end{genericthm*}}

\newtheorem{proposition}[theorem]{Proposition}

\newtheorem*{definition*}{Definition}

\newtheorem*{conjecture*}{Conjecture}
\newtheorem*{theorem*}{Theorem}
\newtheorem*{remark*}{Remark}
\numberwithin{equation}{section}

\DeclareMathOperator{\ord}{ord}

\newcommand{\BZ}{\mathbb Z}

\newcommand{\BQ}{\mathbb Q}
\newcommand{\BF}{\mathbb F}

\providecommand{\abs}[1]{\left\lvert#1\right\rvert}

\makeatletter
\let\@@pmod\pmod
\DeclareRobustCommand{\pmod}{\@ifstar\@pmods\@@pmod}
\def\@pmods#1{\mkern4mu({\operator@font mod}\mkern 6mu#1)}
\makeatother

\title[A lower bound for the two-variable Artin conjecture]{A lower bound for the two-variable Artin conjecture and prime divisors of recurrence sequences}
\date{\today}

\author{M. Ram Murty}
\address{Department of Mathematics, Queen's University, Kingston, Ontario K7L 3N6, Canada.}
\email{murty@mast.queensu.ca}
\thanks{Research of the first and third author partially supported by an NSERC Discovery grant. Research of the third author supported in part by the Canada Research Chairs Program.}

\author{François Séguin}
\address{Department of Mathematics, Queen's University, Kingston, Ontario K7L 3N6, Canada.}
\email{francois.seguin@queensu.ca} 

\author{Cameron L. Stewart}
\address{Department of Pure Mathematics, University of Waterloo, Waterloo, Ontario N2L 3G1, Canada.}
\email{cstewart@uwaterloo.ca}


\begin{document}

\begin{abstract}
In 1927, Artin conjectured that any integer other than -1 or a perfect square generates the multiplicative group $\BZ/p\BZ^\times$ for infinitely many $p$. In \cite{MoSt}, Moree and Stevenhagen considered a two-variable version of this problem, and proved a positive density result conditionally to the generalized Riemann Hypothesis by adapting a proof by Hooley for the original conjecture (\cite{Ho}). In this article, we prove an unconditional lower bound for this two-variable problem. In particular, we prove an estimate for the number of distinct primes which divide one of the first $N$ terms of a non-degenerate binary recurrence sequence. We also prove a weaker version of the same theorem, and give three proofs that we consider to be of independent interest. The first proof uses a transcendence result of Stewart \cite{Stew}, the second uses a theorem of Bombieri and Schmidt on Thue equations \cite{BoSc} and the third uses Mumford's gap principle for counting points on curves by their height \cite{Mum}. We finally prove a disjunction theorem, where we consider the set of primes satisfying either our two-variable condition or the original condition of Artin's conjecture. We give an unconditional lower bound for the number of such primes.
\end{abstract}

\maketitle
\section{Introduction}
In this article we study the two-variable analogue of Artin's conjecture on primitive roots. Artin's original conjecture suggested that for any integer $a$ other than -1 and perfect squares, there are infinitely many primes $p$ for which $a$ generates the multiplicative group $\BZ/p\BZ^\times$. Specifically, Artin conjectured that the set
$$N_a(X) = \#\left\{p\leq X \text{ prime }: \langle a \bmod{p} \rangle = \BZ/p\BZ^\times\right\}$$
has positive density in the set of all primes.
We can trace the origin of this problem all the way back to Gauss. It was apparently popular at the time to study decimal expansions of certain rational numbers. In his Disquisitiones Arithmeticae, Gauss describes the period of the decimal expansion of $\frac{1}{p}$ in terms of the order of $10 \bmod{p}$. Only other such specific cases of this were considered before 1927, at which time Artin formulated the above conjecture.

As of now, the conjecture is still open. There is actually no $a$ for which we know $N_a$ to be infinite. However, there have been major partial results since, the conditional proof by Hooley \cite{Ho} under the assumption of the generalized Riemann Hypothesis being among the most important, as are the works of Gupta and Murty \cite{GuMu} and Heath-Brown \cite{He}. For example, we know that given three mutually coprime numbers $a,b, c$, there are infinitely many primes $p$ for which at least one of $a,b,c$ is a primitive root mod $p$.

Many variations on Artin's original conjecture have since been studied. Moree and Stevenhagen \cite{MoSt} considered a two-variable variant where the set of interest is
$$S = \left\{p \text{ prime } : b \bmod{p} \in \langle a \bmod{p}\rangle \subseteq \BZ / p\BZ^\times \right\}$$
for given $a$ and $b$. They adapted Hooley's argument, as well as using some work by Stephens (\cite{Step}), to show a positive density result for such primes, conditionally under the generalized Riemann Hypothesis. In this article, we prove an unconditional lower bound on the number of primes in this set. Specifically, we prove the following result.

\begin{theorem}\label{thm2}
Let $a,b\in \BZ^\ast$ with $\abs{a}\neq 1$. Then,
$$\left| \left\{p \leq x\text{ prime }: b\bmod{p} \in \langle a \bmod{p} \rangle \subset \BF_p^\ast \right\}\right| \gg \log x. $$
\end{theorem}

We do so by proving in section \ref{sec:RS} a more general result about binary recurrence sequences.
\begin{restatable*}{theorem}{brs}
\label{thm:brs}
Let $\left\{u_n\right\}_{n=1}^{\infty}$ be a non-degenerate binary recurrence sequence with the $n$-th term given by (\ref{eq:brs}). Let $\epsilon$ be a positive real number.  There exists an effectively computable positive number $C$, depending on $\epsilon, a, b, \alpha$ and $\beta$, such that if $N$ exceeds $C$ then
$$\omega \left( \prod_{\substack{ n=1\\u_n\neq 0}}^N u_n \right) > \left( 1- 1/\sqrt{2} - \epsilon\right) N.$$
\end{restatable*}

We also prove a more precise result for the specific case of Lucas sequences.
\begin{restatable*}{theorem}{brsluc}
\label{thm:brs1}
Let $\left\{t_n\right\}_{n=1}^{\infty}$ be a non-degenerate Lucas sequence. Then
\begin{align*}
\omega \left( \prod_{\substack{ n=1}}^N t_n \right) \geq N-9.
\end{align*}
Equality holds when $t_n$ satisfies
\begin{align*}
t_n=t_{n-1}-2t_{n-2} \,\,\,for \,\,\,n=2,3,...
\end{align*}
and $N = 30, 31, 32, 33 \,\text{ or }\,\,34$.
\end{restatable*}

We finally conjecture the following stronger statement.

\begin{restatable*}{conjecture}{conjbrs}
There exist positive numbers $C_1$ and $C_2$, which depend on $a,b,\alpha$ and $\beta$, such that if $\left\{u_n\right\}_{n=1}^{\infty}$ is a non-degenerate binary recurrence sequence, then
\begin{align*}
C_1N \log N \leq \omega \left( \prod_{\substack{ n=1\\u_n\neq 0}}^N u_n \right) \leq C_2N \log N.
\end{align*}
\end{restatable*}
Note that the lower bound obtained from this conjecture could be used to improve Theorem \ref{thm2} by replacing $\log x$ with $\log x \log\log x$ in the lower bound.

We shall also give several proofs, which we believe to be of independent interest, for the following theorem, which is a weaker version of Theorem \ref{thm2}.
\begin{theorem} \label{thm1}
Let $a,b\in \BZ^\ast$ with $\abs{a}\neq 1$. Then,
$$\left| \left\{p \leq x\text{ prime }: b\bmod{p} \in \langle a \bmod{p} \rangle \subset \BF_p^\ast \right\}\right| \gg \log \log x. $$
\end{theorem}

The last theorem we prove is a disjunction theorem.

\begin{theorem} \label{thmdis}
Let $a,b\in \BZ^\ast$ with $(a,b)=1$. Then,
\begin{align*}
\left| \left\{p \leq x\text{ prime }: \begin{array}{c} b\bmod{p} \in \langle a \bmod{p} \rangle \subset \BF_p^\ast \\ \text{or } \langle b \bmod{p} \rangle = \BF_p^\ast \end{array} \right\}\right| \gg \frac{x}{\left(\log x\right)^2}.
\end{align*}
\end{theorem}

This theorem suggests that it might be possible to prove positive density of this set unconditionally.
It is worth noting that unlike the original Artin conjecture, the set $S$ is known to be infinite. Moree and Stevenhagen included in \cite{MoSt} a modification of a simple argument by Pólya (found in \cite{Po}) that proves the infinitude unconditionally. However, trying to derive a lower bound directly from their argument does not lead to any explicit function going to infinity.

We will start by proving Theorem \ref{thm:brs} in section \ref{sec:rs} after a few preliminaries in sections \ref{sec:RS} and \ref{sec:prelim}. Theorem \ref{thm:brs1} will be proven in section \ref{sec:rs88}. Then, we will use Theorem \ref{thm:brs} to prove Theorem \ref{thm2} in section \ref{sec:thm2}. Our three proofs for Theorem \ref{thm1} are in sections \ref{sec:rs8}, \ref{sec:BS} and \ref{sec:mgp} respectively. Finally, we will prove Theorem \ref{thmdis} in section \ref{sec:thmdis}.

\section{Prime divisors of terms of recurrence sequences}\label{sec:RS}

For any non-zero integer $n$ let $\omega (n)$ denote the number of distinct prime factors of $n$.
Let $r$ and $s$ be integers with $r^2+4s\neq 0$. Let $u_0$ and $u_1$ be integers and put
$$u_n = ru_{n-1}+su_{n-2} \text{ for } n \geq 2.$$
Then,
\begin{align}\label{eq:brs}
u_n = a\alpha^n + b\beta^n
\end{align}
where $\alpha$ and $\beta$ are the roots of the polynomial
$$x^2-rx-s$$
and
\begin{align*}
a = \frac{u_0\beta - u_1}{\beta-\alpha}, \qquad b=\frac{u_1-u_0\alpha}{\beta-\alpha}.
\end{align*}

The sequence $\left\{u_n\right\}_{n=0}^{\infty}$ is called a \emph{binary recurrence sequence}. It is said to be \emph{non-degenerate} if $ab\alpha\beta \neq 0$ and $\alpha/\beta$ is not a root of unity.

Note that for non-degenerate binary recurrence sequences, if $\abs{\alpha} \geq \abs{\beta}$, then
\begin{align}\label{eq:BRS}
\abs{\alpha}\geq \sqrt{2}.
\end{align}
Indeed, if $\alpha$ and $\beta$ are integers this is obvious. Otherwise, if $\abs{s} \geq 2$, then $2\leq \abs{s} = \abs{\alpha\beta} \leq \abs{\alpha}^2$. Finally, suppose $\abs{s}=1$. If $\BQ(\alpha)$ was an imaginary quadratic field, $\frac{\alpha}{\beta}$ would be a root of unity, which is not allowed. If $\BQ(\alpha)$ is totally real, then $\alpha = a+b\sqrt{D}$ and $\beta = \pm(a-b\sqrt{D})$ for some $a,b\in \BZ\left[\frac{1}{2}\right]$. Since $\abs{\alpha} \geq \abs{\beta}$, $a$ and $b$ have the same sign, and so $\abs{\alpha} = \abs{a} + \abs{b}\sqrt{D} \geq \sqrt{2}$.

In 1921 Polya \cite{Po} showed that 
\begin{align}\label{eq:Po}
{\omega \left( \prod_{\substack{ n=1\\u_n\neq 0}}^N u_n \right) \to\infty} \,\, as \,\, {N \to\infty};
\end{align}
Gelfond \cite{Gel} and Mahler \cite{Mah} in 1934 and Ward \cite{War} in 1954 gave alternative proofs of (\ref{eq:Po}).
In 1987 Shparlinski \cite{Shpa} showed that
\begin{align}\label{eq:Pol}
{\omega \left( \prod_{\substack{ n=1\\u_n\neq 0}}^N u_n \right) \gg N/ \log N},
\end{align}
improving on an earlier result of his \cite{Sha2} where he had established (\ref{eq:Pol}) with the righthand side of (\ref{eq:Pol}) replaced by $N^{1/2}$. It should be noted that Shparlinski's result (\ref{eq:Pol}) applies not just to binary recurrence sequences but to non-degenerate sequences of order $k$ with $k \geq 2$.

We are able to improve upon (\ref{eq:Pol}) for binary recurrence sequences.

\brs

Theorem \ref{thm:brs} is the key result we need to establish Theorem \ref{thm2}.
\medskip

A \emph{Lucas sequence} is a non-degenerate binary recurrence sequence $\left\{t_n\right\}_{n=0}^{\infty}$ with $t_0=0$ and $t_1=1$. Thus, $a=\frac{1}{\alpha-\beta}$ and $b=\frac{-1}{\alpha - \beta}$, so that from (\ref{eq:brs}), we have
\begin{align}\label{eq:LUC10}
t_n = \frac{\alpha^n - \beta^n}{\alpha-\beta}
\end{align}
for $n\geq 0.$
The divisibility properties of Lucas sequences have been extensively studied, see for example \cite{Lucas}, \cite{Ca} and \cite{Stew78}, and for these binary recurrence sequences we are able to improve on Theorem \ref{thm:brs}.

\brsluc

It is not difficult to show that if $\left\{u_n\right\}_{n=1}^{\infty}$ is a non-degenerate binary recurrence sequence then
\begin{align}\label{eq:LUC2}
\omega \left( \prod_{\substack{ n=1\\u_n\neq 0}}^N u_n \right) \ll N^2/ \log N.
\end{align}
To see this suppose that $u_n$ is given by (\ref{eq:brs}) with $\abs{\alpha} \geq \abs{\beta}$. Then
\begin{align*}
|u_n| \leq (|a|+|b|)|\alpha|^n
\end{align*}
and therefore
\begin{align}\label{eq:LUC3}
\left| \prod_{\substack{ n=1\\u_n\neq 0}}^N u_n \right| \leq (|a|+|b|)^N|\alpha|^{N(N+1)/2}.
\end{align}
Let $2=p_1,p_2,... $ be the sequence of prime numbers. By the Prime Number Theorem
\begin{align}\label{eq:LUC4}
  \prod_{\substack{ i=1}}^t p_i = \exp((1+o(1))t\log t .
\end{align}
Observe that if
\begin{align*}
 \prod_{i=1}^t p_i  \geq  \left| \prod_{\substack{ n=1\\u_n\neq 0}}^N u_n \right| .
\end{align*}
then
\begin{align*}
\omega \left( \prod_{\substack{ n=1\\u_n\neq 0}}^N u_n \right) \leq t.
\end{align*}
Thus (\ref{eq:LUC2}) follows from (\ref{eq:BRS}), (\ref{eq:LUC3}) and (\ref{eq:LUC4}).

We make the following conjecture.
\conjbrs

\section{Preliminaries for the proof of Theorem \ref{thm:brs}} \label{sec:prelim}
The first two results we require concern prime divisors of Lucas numbers.
\begin{proposition}\label{prop:1}
Let $\left\{t_n\right\}_{n=0}^{\infty}$ be a Lucas sequence, as in (\ref{eq:LUC10}), with $\abs{\alpha} \geq \abs{\beta}$. If $p$ is a prime number not dividing $\alpha\beta$ then $p$ divides $t_n$ for some positive integer $n$ and if $\ell$ is the smallest such $n$ then
\begin{align*}
\frac{\log p - \frac{\log 2}{2}}{\log \abs{\alpha}} \leq \ell \leq p+1.
\end{align*}
\end{proposition}
\begin{proof}
Apart for the lower bound, this is Lemma 7 of \cite{Stew}. The lower bound follows from $p\leq \abs{t_\ell} \leq \sqrt{2}\abs{\alpha}^\ell$.
\end{proof}

For any rational number $x$ let $\abs{x}_p$ denote the p-adic value of $x$, normalized so that $\abs{p}_p = p^{-1}.$
\begin{proposition}\label{prop:2}
Let $\left\{t_n\right\}_{n=0}^{\infty}$ be a Lucas sequence, as in (\ref{eq:LUC10}), with $\alpha +\beta$ and $\alpha\beta$ coprime. Let $p$ be a prime number which does not divide $\alpha\beta$, let $\ell$ be the smallest positive integer for which $p$ divides $t_\ell$ and let $n$ be a positive integer. If $\ell$ does not divide $n$  then
\begin{align*}
|t_n|_p=1.
\end{align*}
If $n= \ell k$ for some positive integer $k$ we have, for $p > 2$,
\begin{align*}
\abs{t_n}_p &= \abs{t_\ell}_p \abs{k}_p,
\intertext{while for $p=2$,}
\abs{t_n}_2 &= \begin{cases}\abs{t_\ell}_2 & \text{ for } k \text{ odd }\\ 2\abs{t_{2\ell}}_2 \abs{k}_2 & \text{ for } k \text{ even. }\end{cases}
\end{align*}
\end{proposition}
\begin{proof}
This is Lemma 8 of \cite{Stew} and it is based on work of
 Carmichael  \cite{Ca}, see also \cite{Stew78}.
\end{proof}

In addition to the results about Lucas sequences we need an estimate from below for the size of the $n$-th term of a non-degenerate binary recurrence sequence.
\begin{proposition}\label{prop:3}
Let $u_n$ be the $n$-th term of a non-degenerate binary recurrence sequence as in (\ref{eq:brs}). There exist positive numbers $c_0$ and $c_1$, which are effectively computable in terms of $a$ and $b$, such that for all $n > c_1$,
$$\abs{u_n} \geq \abs{\alpha}^{n-c_0\log n}.$$
\end{proposition}
\begin{proof}
This is Lemma 6 in \cite{Stew} and is a consequence of Baker's theory of linear forms in logarithms.
\end{proof}

\section{The proof of Theorem \ref{thm:brs} }\label{sec:rs}
It suffices to prove the result under the assumption that $\alpha+\beta$ and $\alpha\beta$ are coprime or, equivalently, that $r$ and $s$ are coprime. We shall also suppose, without loss of generality, that
\begin{align*}
\abs{\alpha}\geq \abs{\beta}.
\end{align*}
In the following discussion, every $c_i$ will denote a positive number effectively computable in terms of $a,b, \alpha$ and $\beta$. For any prime $p$ let $[p]$ denote the principal ideal generated by $p$ in the ring of algebraic integers of $\BQ(\alpha)$. Put
\begin{align*}
a'=(\beta-\alpha)a, \ \ b'=(\beta-\alpha)b.
\end{align*}

Let $p$ be a prime which divides $\alpha\beta$ and let $\mathfrak{p}$ be a prime ideal which divides $[p]$. Then, since $\alpha+\beta$ and $\alpha\beta$ are coprime integers, $\mathfrak{p}$ divides either $[\alpha]$ or $[\beta]$. Thus, by (\ref{eq:brs}) for $m>c_1$ we have
\begin{align}\label{eq:BRS4}
\abs{u_m}_p\geq \abs{a'b'}_p.
\end{align}

It follows from Proposition \ref{prop:3} that $u_m$ is non-zero for $m>c_2$. Put
\begin{align*}
\gamma=1-1/\sqrt{2}.
\end{align*}
Then $\gamma N$ exceeds both $c_1$ and $c_2$ for $N>c_3$. For each positive integer $N$ with $N>c_3$ put
\begin{align*}
S = S(N) \coloneqq \prod_{\gamma N < n \leq N} u_n.
\end{align*}
Our proof proceeds by a comparison of estimates for $S$.

By Proposition \ref{prop:3} there exists $c_4$ such that 
\begin{align*}
\abs{S} &\geq \prod_{\gamma N < n \leq N} \abs{\alpha}^{n-c_4\log n}
\end{align*}
and so
\begin{align}\label{eq:dag3}
\abs{S} & \geq \abs{\alpha}^{\frac{(1-\gamma^2)N^2}{2} - c_5N\log N}.
\end{align}

Plainly,
\begin{align*}
\abs{S} = \prod_{p | S} \abs{S}_p^{-1}.
\end{align*}
We first estimate $\abs{S}_p^{-1}$ for primes $p$ which divide $\alpha\beta$. By (\ref{eq:BRS4}) we have
\begin{align*}
\abs{S}_p^{-1} \leq \abs{a^\prime b^\prime}_p^{-N}.
\end{align*}

We shall now estimate $\abs{S}_p^{-1}$ for primes $p$ which divide $S$ but do not divide $\alpha\beta$. For each such prime $p$ we let $n(p)$ be the smallest integer with $\gamma N < n(p) \leq N$ for which
$$\abs{u_{n(p)}}_p \leq \abs{u_n}_p \qquad \text{ for }\gamma N<n\leq N.$$
For positive integers $m$ and $r$ with $m\geq r$,
\begin{align}\label{eq:2stars}
u_m - \beta^{r} u_{m-r} =  a^\prime \alpha^{m-r} t_r.
\end{align}
Let $\abs{\ \ }_p$ denote an extension of $\abs{\ \ }_p$ from $\BQ$ to $\BQ(\alpha)$. For each integer $r$ with $1\leq r <n(p)-\gamma N$
\begin{align*}
|a'b't_r|_p \leq |a't_r|_p=|a'\alpha^{n(p)-r}t_r|_p
\end{align*}
and, by (\ref{eq:2stars}) with $m=n(p)$,
\begin{align*}
|a'\alpha^{n(p)-r}t_r|_p \leq \max(|u_{n(p)}|_p , |\beta^ru_{n(p)-r}|_p)
\end{align*}
Since $|\beta|_p =1$ ,
\begin{align*}
\max(|u_{n(p)}|_p , |\beta^ru_{n(p)-r}|_p)= \max(|u_{n(p)}|_p , |u_{n(p)-r}|_p)=|u_{n(p)-r}|_p
\end{align*}
and we deduce that
\begin{align*}
|a'b't_{r} |_p \leq |u_{n(p)-r}|_p
\end{align*}
for $1 \leq r < n(p)-\gamma N.$ Hence,
$$\abs{\prod_{\gamma N < n < n(p)} u_n}_p \geq \prod_{1\leq r < n(p)-\gamma N} \left(\abs{t_r}_p \abs{a^\prime b^\prime}_p \right).$$

Letting $\ell = \ell(p)$ be the smallest integer for which $p | t_\ell$, we have by Proposition \ref{prop:1} and Proposition \ref{prop:2} that if $p>2$,
\begin{align*}
\prod_{1\leq r< n(p) - \gamma N} \abs{t_r}_p &= \abs{t_\ell}_p^{s_1} \abs{s_1!}_p
\intertext{where $s_1 = \left\lfloor \frac{n(p) - \gamma N}{\ell} \right\rfloor$, while for $p=2$}
\prod_{1\leq r< n(2) - \gamma N} \abs{t_r}_2 &= \abs{t_\ell}_2^{s_1} \abs{\frac{t_{2\ell}}{t_\ell}}_2^{s_2} \abs{s_2!}_2
\end{align*}
with $s_2 = \left\lfloor \frac{n(2) -\gamma N}{2\ell} \right\rfloor$.

Next, on setting $m-r= n(p)$ and letting $r$ run over those integers such that $n(p) + r \leq N$, we find that for $p>2$
\begin{align*}
\prod_{n(p) < n \leq N} \abs{u_n}_p &\geq \abs{t_\ell}_p^{s_3} \abs{s_3!}_p \abs{a^\prime b^\prime}_p^{N-n(p)}
\intertext{while for $p=2$}
\prod_{n(2) < n \leq N} \abs{u_n}_2 &\geq \abs{t_\ell}_2^{s_4}\abs{\frac{t_{2\ell}}{t_\ell}}_2^{s_4} \abs{s_4!}_2 \abs{a^\prime b^\prime}_p^{N-n(2)}
\end{align*}
where
$$s_3 = \left\lfloor \frac{N- n(p)}{\ell} \right\rfloor \quad \text{ and } \quad s_4 = \left\lfloor \frac{N - n(2)}{2\ell} \right\rfloor.$$

Putting all this together gives, for $p>2$,
\begin{align*}
\abs{S}_p^{-1} &\leq \abs{t_\ell}_p^{-s} \abs{s!}_p^{-1} \abs{a^\prime b^\prime}_p^{-N} \abs{u_{n(p)}}_p^{-1}
\intertext{where $s = \left\lfloor \frac{N-\gamma N}{\ell} \right\rfloor$. As $\abs{t_\ell}_p^{-1} \leq \abs{t_\ell} \leq 2 \abs{\alpha}^\ell$, we find that}
\abs{S}_p^{-1} &\leq 2^{\frac{N}{\ell(p)}} \abs{\alpha}^{N-\gamma N} \abs{N!}_p^{-1} \abs{a^\prime b^\prime}_p^{-N} \abs{u_{n(p)}}_p^{-1}
\intertext{for $p>2$. For $p=2$ we have}
\abs{S}_2^{-1} &\leq 4^{\frac{N}{\ell(2)}} \abs{\alpha}^{2(N-\gamma N)} \abs{N!}_2^{-1} \abs{a^\prime b^\prime}_2^{-N} \abs{u_{n(2)}}_2^{-1}.
\end{align*}

Putting $T = \omega(S)$, we may suppose $T < N$ for otherwise we are done. Inserting the above estimates, we obtain
\begin{align}\label{eq:ddag}
S = \prod_{p | S} \abs{S}_p^{-1} \leq \left( \prod_{p|S} 4^{\frac{N}{\ell(p)}}\right) \abs{\alpha}^{(N-\gamma N) (T+1)} N! \abs{a^\prime b^\prime}^N \prod_{p|S} \abs{u_{n(p)}}_p^{-1}.
\end{align}

We need to estimate the right hand side and compare it with (\ref{eq:dag3}).
\begin{align*}
\prod_{p|S} 4^{\frac{N}{\ell(p)}} &\leq \prod_{\substack{p|S\\p < T/\log T} } 4^N \cdot \prod_{\substack{p|S\\p>T/\log T}} 4^{\frac{N}{\ell(p)}}\\
&\leq 4^{NT/\log T} \cdot \prod_{\substack{p|S\\p> T/\log T}} 4^{\frac{N}{\ell(p)}}.
\end{align*}
However, by Proposition \ref{prop:1},
$$\ell(p) \geq \frac{\log p - \log 2}{\log \abs{\alpha}} > \frac{\log T - \log \log T - \log 2}{\log \abs{\alpha}}.$$
As $\abs{\alpha} \geq \sqrt{2}$, we deduce
$$\prod_{p|S} 4^{\frac{N}{\ell(p)}} < e^{c_8 N^2 / \log N}.$$
Inserting this in inequality (\ref{eq:ddag}) and using $N! \leq N^N$, we get
\begin{align*}
\prod_{p | S} \abs{S}_p^{-1} < e^{c_9 N^2 / \log N} \abs{\alpha}^{N(1-\gamma)T} \prod_{p|S} \abs{u_{n(p)}}_p^{-1}.
\end{align*}

For each $n$, we have $\abs{u_n} \leq (\abs{a}+\abs{b})\abs{\alpha}^n$, since $\abs{\alpha} \geq \abs{\beta}$. Put
$$K \coloneqq \left\{n(p) : p| S\right\}.$$
Then, $\abs{K} \leq T$. Thus,
\begin{align*}
\prod_{p|S} \abs{u_{n(p)}}_p^{-1} &\leq \prod_{k\in K} \abs{u_k}\\
 &\leq \prod_{k\in K} \left(\abs{a} + \abs{b} \right) \abs{\alpha}^k\\
&\leq \left(\abs{a}+\abs{b} \right)^{T} \abs{\alpha}^{NT - \frac{T(T-1)}{2}}.
\end{align*}
 Putting everything together, we get
$$\prod_{p | S} \abs{S}_p^{-1} \leq e^{c_{10} N^2 / \log N} \abs{\alpha}^{(2-\gamma)NT-\frac{T^2}{2}}$$
and as $\abs{\alpha} \geq \sqrt{2}$, we get from (\ref{eq:dag3})
$$\abs{\alpha}^{\frac{N^2(1-\gamma^2)}{2}} < e^{c_{11} N^2 / \log N} \abs{\alpha}^{(2-\gamma)NT-\frac{T^2}{2}}.$$
Therefore $T>(1-1/\sqrt{2} - \epsilon) N$  for $N> c_{12}$ since the roots of the quadratic $x^2-(4-2\gamma)x+1-\gamma^2$ are $\gamma$ and $\gamma +2 \sqrt{2}$.

\section{The proof of Theorem \ref{thm:brs1} }\label{sec:rs88}
Let  $\left\{t_n\right\}_{n=1}^{\infty}$ be a non-degenerate Lucas sequence with $n$-th term given by (\ref{eq:LUC10}). We may assume, without loss of generality that $\alpha+\beta$ and $\alpha\beta$ are coprime.  A primitive divisor of $t_n$ is a prime $p$ which divides $t_n$ but does not divide $(\alpha-\beta)^2t_2 \cdots t_{n-1}.$ In \cite{Stew2} Stewart showed that there are only finitely many Lucas sequences , with $\alpha+\beta$ and $\alpha\beta$ coprime, for which $t_n$ does not possess a primitive divisor when $n>4$ and $n\neq 6$ and these sequences may be explicitly determined. It then follows that the number of distinct prime factors of $ \prod_{\substack{ n=1}}^N t_n $ is at least $N-5$ whenever $\left\{t_n\right\}_{n=1}^{\infty}$ is not an exceptional sequence. Bilu, Hanrot and Voutier \cite{BHV} determined the complete list of exceptional sequences and by examining the list we see that whenever $\left\{t_n\right\}_{n=1}^{\infty}$ is a non-degenerate Lucas sequence
\begin{align*}
\omega \left( \prod_{\substack{ n=1}}^N t_n \right) \geq N-9,
\end{align*}
with equality holding when $t_n$ satisfies
\begin{align*}
t_n=t_{n-1}-2t_{n-2} \,\,\,for \,\,\,n=2,3,...
\end{align*}
and $N = 30, 31, 32, 33 \,\text{ or }\,\,34$.

\section{Proof of Theorem \ref{thm2}} \label{sec:thm2}
First, notice that the set of interest
\begin{align*}
S_x &= \left\{p \leq x\text{ prime }: b\bmod{p} \in \langle a \bmod{p} \rangle \subset \BF_p^\ast \right\}
\intertext{can be expressed as}
S_x &= \{p \leq x \text{ prime }: p| (a^n-b) \text{ for some }n\}.
\end{align*}
Suppose that $p$ divides $a^n-b$ with $n\leq \left\lfloor \frac{\log x}{\log a}\right\rfloor =:N$. Then, $p \leq a^n-b < a^n \leq x$.

Therefore, it is clear that 
$$\#S_x \gg \#\{p \text{ prime } : p| (a^n-b) \text{ for some }n \leq N\}.$$

Consider the binary recurrence sequence given by $u_n = a^n-b$ (here $\alpha,\beta, a$ and $b$ in (\ref{eq:brs}) are respectively $a$, 1, 1 and $b$). Then, by Theorem \ref{thm:brs},
$$\#\left\{ p : p | a^n-b \text{ for some } n\leq N\right\} \gg N$$
for $N = \left\lfloor \frac{\log x}{\log \abs{a}} \right\rfloor$, and so 
$$\#\left\{ p \leq x : p | a^n-b \text{ for some } n\right\} \gg \log x.$$

\section{Theorem \ref{thm1} via the greatest prime factor of terms of recurrence sequences }\label{sec:rs8}
The first proof uses the following result by Stewart about the growth of the largest prime divisor in a type of recurrence sequence.

 For any integer $n$ let $P(n)$ denote the greatest prime factor of $n$ with the convention that $P(0)=P(1)=P(-1).$
\begin{theorem}[Stewart \cite{Stew19}] \label{thm:stewart}
Let $u_n$, as in (\ref{eq:brs}), be the $n$-th term of a non-degenerate binary recurrence sequence. There exists a positive number $C$, which is effectively computable in terms of $a,b,\alpha$ and $\beta$, such that for $n>C$
$$P(u_n)> \sqrt{n}\exp( \log n/ 104\log \log n).$$
\end{theorem}
We actually need a special case of this result. Note that for $\alpha = 1$, $x=a$, $\beta = b$ and $y=1$, the above theorem yields
$$P(a^n-b) \gg_{a,b} \sqrt{n}\exp( \log n/ 104\log \log n).$$
This is what we will be using.

\begin{proof}[Proof of Theorem \ref{thm1}]
We will prove the theorem for the case $a,b > 0$ for simplicity. The proof can be easily adapted to the general case. See the remark for more details. Again, let
$$S_x = \left\{p \leq x\text{ prime }: b\bmod{p} \in \langle a \bmod{p} \rangle \subset \BF_p^\ast \right\}.$$

Using the same argument as in section \ref{sec:thm2}, we have
$$\#S_x \gg \#\{p \text{ prime } : p| (a^n-b) \text{ for some }n \leq N\}.$$
for $N \coloneqq \left\lfloor \frac{\log x}{\log a}\right\rfloor$.

Consider the sequence $\xi_n = a^n-b$ for $N-y \leq n \leq N$ where $y$ is a parameter to be chosen later. As noted above, $p| \xi_n$ in this range implies $p\leq x$. Now consider $P(a^n-b)$, the largest prime factor of $a^n-b$, for each of those $n$. Those yield $y$ primes, albeit a priori not necessarily distinct.

Suppose that for some $m,n$ with $N-y \leq m < n \leq N$ we have $$P(a^n-b) = P(a^m-b) =:q.$$ Then,
\begin{align*}
a^n&\equiv b \bmod q\\
a^m&\equiv b \bmod q\\
\Rightarrow a^{nm} &\equiv b^n \equiv b^m \bmod q
\end{align*}
and thus $q$ divides $b^n - b^m$.

From Theorem \ref{thm:stewart}, we know that $q$ exceeds $b$ for $x$ large enough, and so $q$ does not divide $b$. We conclude that $q | (b^{n-m}-1)$. In particular, we have that $q \leq b^{n-m} -1 < b^{n-m}$. However, $n-m \leq y$, and so choosing $y = \frac{\log \left(C_1 \sqrt{N}\right)}{\log b}$ yields
$$P(a^n-b) = q < b^{n-m} \leq C_1 \sqrt{N}$$
which is a contradiction to Theorem \ref{thm:stewart} for properly chosen $C_1$.

We therefore have $y$ distinct primes in the set $S_x$, where
$$y = \frac{1}{2 \log b} \left(\log \log x - \log\log\log x\right) + C^\prime \gg \log\log x.$$
\end{proof}

\section{Theorem \ref{thm1} via Thue equations}\label{sec:BS}
The second proof of Theorem \ref{thm1} uses a result on Thue equations. Recall that a Thue equation is an equation of the form
$$F(x,y) = h$$
where $F(x,y) = a_0 x^r + a_1 x^{r-1}y + \cdots + a_r y^r$ is an integral binary form of degree at least $3$. We therefore have the following result for the number of solutions to such an equation.

\begin{theorem}[Bombieri, Schmidt \cite{BoSc}]
Let $F(x,y)$ be an irreducible binary form of degree $r\geq 3$ with rational integral coefficients. The number of primitive solutions of the equation
$$\abs{F(x,y)} =h$$
does not exceed
$$c_1 r^{t+1}$$
where $c_1$ is an absolute constant and $t$ is the number of distinct prime factors of $h$.
\label{thm:B-S}
\end{theorem}

We now proceed with our second proof of Theorem \ref{thm1}. For this particular proof, we require the extra condition that $a$ and $b$ are coprime. However, this condition is not too restrictive and we believe the proof to still have its merits.

\begin{proof}[Proof of Theorem \ref{thm1}]
Suppose that $(a,b)=1$. As in the previous proof, notice that 
$$S_x = \{p \leq x \text{ prime }: p| (a^n-b) \text{ for some }n\}.$$
Fix $x$. Then, again,
\begin{align}\label{eq:k1}
\#S_x \gg \#\{p \text{ prime } : p| (a^n-b) \text{ for some }n \leq N\}
\end{align}
where $N:=\left\lfloor\frac{\log x}{\log a}\right\rfloor$. Denote by $k$ the quantity on the right hand side of (\ref{eq:k1}).

Since there are at most $k$ primes dividing the numbers $a^n-b$ with $n$ varying, we can write
$$a^n-b = p_1^{\alpha_1(n)}p_2^{\alpha_2(n)}\cdots p_k^{\alpha_k(n)}$$
with $p_i$ distinct primes, and $\alpha_i(n) = \ord_{p_i}(a^n-b)$.

For every fixed $n$, we have
\begin{align*}
a^n - p_1^{\alpha_1(n)}p_2^{\alpha_2(n)}\cdots p_k^{\alpha_k(n)} &= b\\
a^{\delta} a^{3j} - p_1^{\epsilon_1}\cdots p_k^{\epsilon_k} p_1^{3j_1}\cdots p_k^{3j_k} &= b\\
\intertext{where $\delta$ and $\epsilon_i$ are the residue of $n$ and $\alpha_i(n)$ modulo 3 respectively ($\delta, \epsilon_i\in \{0,1,2\}$), so we obtain the equation}
a^{\delta} \left(a^{j}\right)^{3} - \left(p_1^{\epsilon_1}\cdots p_k^{\epsilon_k}\right) \left(p_1^{j_1}\cdots p_k^{j_k}\right)^{3} &= b.
\end{align*}
As $n$ varies, we obtain at most $3^{k+1}$ different equations of the form
$$a^{\delta} X^3 - \left(p_1^{\epsilon_1}\cdots p_k^{\epsilon_k}\right) Y^3 = b.$$
Also, every single $n\leq N$ gives a different solution to one of those equations. All the solutions are primitive since $(a,b)=1$. Therefore, one of them has at least $\frac{N}{3^{k+1}}$ solutions.

Let $C = c_1 3^{1+t}$ where $t$ is the number of prime factors of $b$, and $c_1$ is the constant appearing in Theorem \ref{thm:B-S}. Then, $\frac{N}{3^{k+1}}>C$ would be a contradiction to Theorem \ref{thm:B-S}, and so we have that 
\begin{align*}
\frac{N}{3^{k+1}} &\leq C\\
\Rightarrow N &\ll 3^{k}\\
\Rightarrow \log N &\ll k.
\intertext{Recall from the definition of $N$ that $N\gg \log x$, and so}
\log \log x &\ll_{a,b} k
\end{align*}
which completes the proof.
It is worth noting that the dependence on $a$ and $b$ can easily be made explicit as
$$k \gg \log \log x - \log \log a - \omega(b)$$
where $\omega(b)$ denotes the number of distinct prime factors of $b$, and the implicit constant is absolute.
\end{proof}

\section{Theorem \ref{thm1} via Mumford's gap principle}\label{sec:mgp}
This proof uses Mumford's theorem about counting points on curves using a height 
function.
\begin{theorem}[Mumford \cite{HiSi}, \cite{Mum}]\label{mumthm}
Let $C/K$ be a curve of genus $g\geq 2$ defined over a number field. Then there is a constant $c$ depending on $C/K$ and the height function $H$ used, such that
$$\#\{P\in C(K) : H(P) \leq T\} \leq c \log\log T$$
for all $T\geq e^{e}$, where $H$ is a fixed multiplicative height function on $C$.
\end{theorem}
It is important to note that we can make the constant $c$ in Theorem \ref{mumthm} depend only on the field $\overline{K}$. As such, we can apply the theorem to quadratic twists of the same curve with the same constant for each of them. See \cite[Lemma 5]{LeMu} for a proof of this fact.

\begin{proof}[Proof of Theorem \ref{thm1}]
The general idea of this proof is similar to that of section \ref{sec:BS}. As before,
\begin{align} \label{eq:k}
\#S_x \gg \#\{p \text{ prime } : p| (a^n-b) \text{ for some }n \leq N\}
\end{align}
where $N\coloneqq\left\lfloor\frac{\log x}{\log a}\right\rfloor$. Denote by $k$ the quantity on the right hand side of (\ref{eq:k}).

Again, write
$$a^n-b = p_1^{\alpha_1(n)}p_2^{\alpha_2(n)}\cdots p_k^{\alpha_k(n)}$$
with $p_i$ distinct primes, and $\alpha_i(n) = \ord_{p_i}(a^n-b)$.
This time, we consider only the $n$ divisible by 5, and write
\begin{align*}
a^{5j} - p_1^{\epsilon_1}\cdots p_k^{\epsilon_k} p_1^{2j_1}\cdots p_k^{2j_k} &= b\\
\intertext{where $\epsilon_i$ are the residue of $\alpha_i(n)$ modulo 2, so we obtain the equation}
 \left(p_1^{\epsilon_1}\cdots p_k^{\epsilon_k}\right) \left(p_1^{j_1}\cdots p_k^{j_k}\right)^{2}&= \left(a^{j}\right)^{5} - b.
\end{align*}

Now, consider the curve given by the equation
$$C_b: Y^2 = X^5-b.$$
We know this to be a hyperelliptic curve over $\BQ$, and thus a curve of genus $g\geq 2$. Also, if we let $D_n = p_1^{\epsilon_1}\cdots p_k^{\epsilon_k}$, we can consider the quadratic twist
$$C_{b, D_n} : D_n Y^2 =  X^5-b.$$

However, any point $(x,y)$ on this new curve would give
\begin{align*}
D_n y^2 &=  x^5-b\\
(\sqrt{D_n} y)^2 &= x^5 -b
\end{align*}
and so simply amounts to a point on $C_b\left(\BQ\left(\sqrt{D_n}\right)\right)$.

From above, we see that every $n\equiv 0 \bmod{5}$ gives a solution to the curve $C_{b, D_n}$. Since the $X$ coordinate of those points are distinct, it is clear that the points are distinct. As $n$ varies over multiples of 5 between $0$ and $N$, we get $\left\lfloor \frac{N}{5}\right\rfloor$ distinct solutions to at most $2^k$ different curves. It follows that one of these curves has at least $\frac{N}{5\cdot 2^{k}}$ solutions.

Consider the ``naïve'' multiplicative height function on $C_{b,D_n}$ given by $H\left(P\right) = \max\{\abs{x}, \abs{d}\}$ where $P=\left(\frac{x}{d^2},\frac{y}{d^3}\right)$ with $x,y$ and $d$ integers, and $(x,d)=(y,d)=1$.

Then, note that all the solutions produced above for the curves $C_{b,D_n}$ have height at most $a^N$. We then apply Mumford's theorem with this height function to conclude that
\begin{align*}
\#\left\{P\in C_{b,D_n}(\BQ) : H(P) \leq a^N\right\} &\leq c \log\log a^N.\\
\intertext{By the previous comment on quadratic twists,}
\#\left\{P\in C_{b}\left(\BQ\left(\sqrt{D_n}\right)\right) : H(P) \leq a^N\right\} &\leq c \log\log a^N.
\end{align*}
Note that our previous comment about the independence of the constant on the field in Mumford's theorem allows us to have the constant $c$ here be independent of $n$. Hence, by the above
\begin{align*}
\frac{N}{5\cdot 2^{k}} &\leq c \log\log a^N = c \log N + c\log \log a.\\
2^k &\geq \frac{c^\prime N}{\log N + \log\log a}.
\end{align*}

Therefore,
$$ k \gg \log N \gg \log \log x.$$
\end{proof}

We want to point out that even if all three proofs give bounds of the same order of magnitude with respect to $x$, the dependence of the implied constants on $a$ and $b$ vary for each approach. For example, the proof in section \ref{sec:BS} reduces the dependence on $b$ dramatically. Note also that the dependence on $b$ of the implicit constant in section \ref{sec:mgp} is harder to make explicit as the constant given from Mumford's theorem depends on $b$. However, we see that the proof of section \ref{sec:BS} requires an extra condition on $a$ and $b$ to use Theorem \ref{thm:B-S}, albeit a mild one.

In any case, as all three proofs use ideas fundamentally different from each other, we consider that they are of independent interest.

\section{Second order recurrence sequences}
In \cite{MoSt}, Moree and Stevenhagen actually consider the two-variable problem with $a$ and $b$ rational numbers (and then disregard the finitely many primes dividing their numerators or denominators). Here, for clarity, we restricted our attention to integers. However, it is not very hard to retrieve our results in the case where $a$ and $b$ are rational numbers.

Write $a=\frac{a_1}{a_2}$ and $b = \frac{b_1}{b_2}$ with $\gcd(a_1,a_2)=\gcd(b_1,b_2)=1$. Then, the set of primes we are interested in counting
$$S_x = \left\{p \leq x\text{ prime }: b\bmod{p} \in \langle a \bmod{p} \rangle \subset \BF_p^\ast \right\}$$
can be written as
$$S_x = \{p \leq x \text{ prime }: p| (b_2a_1^n-b_1a_2^n) \text{ for some }n\}.$$
The sequence $(b_2a_1^n-b_1a_2^n)$ is a linear recurrence sequence of order 2 and so we may again apply the theorem of Stewart.

For the proof of section \ref{sec:mgp}, it is also easy to generalize the argument. Indeed, following the same notation, we can write for $n\equiv 0 \bmod{10}$
\begin{align*}
b_2a_1^n - b_1a_2^n = p_1^{2j_1+\epsilon_1} \cdots p_k^{2j_k+\epsilon_k}\\
\left(p_1^{\epsilon_1} \cdots p_k^{\epsilon_k}\right)\left(\frac{p_1^{j_1} \cdots p_k^{j_k}}{a_1^{n/2}}\right)^2 = b_2\left(\frac{a_1^{n/5}}{a_2^{n/5}}\right)^{5} - b_1
\end{align*}
which gives the rational solution $\left(\frac{a_1^{n/5}}{a_2^{n/5}},\frac{p_1^{j_1} \cdots p_k^{j_k}}{a_1^{n/2}}\right)$ to the hyperelliptic curve $D_nY^2=b_2X^5-b_1$. Since Mumford's theorem consider any rational solutions, and since the height of these solutions is again at most $\max\{\abs{a_1^{N}},\abs{a_2^{N}}\} \sim x$, the rest of the proof goes through unchanged.

The proof in section \ref{sec:BS} is trickier to generalize. Indeed, the result from Bombieri and Schmidt we use considers only integral solutions to the Thue equation. However, similarly to what we did above, we would need here a bound on the number of $S$-integer solutions to the Thue equation.

\section{Proof of Theorem \ref{thmdis}} \label{sec:thmdis}
This proof mainly relies on the following theorem of Gupta and Murty.
\begin{theorem}[\cite{GuMu}]
Fix $a,b$ coprime integers. There exists a constant $c>0$ such that
$$\#\left\{p\leq x \text{ prime } : p-1 \in 2P_2 \text{ and } \left(\frac{a}{p}\right) = \left(\frac{b}{p}\right) = -1\right\} \geq \frac{c x}{(\log x)^2}$$
where $P_2$ is the set of numbers either prime or product of two primes.
\end{theorem}

\begin{proof}[Proof of Theorem \ref{thmdis}]

We start by considering only the primes in the set
$$T_x = \left\{p\leq x \text{ prime } : p-1 \in 2P_2 \text{ and } \left(\frac{a}{p}\right) = \left(\frac{b}{p}\right) = -1\right\}$$
and ask how many of them are also in our set of interest
\begin{align*}
S^{\prime}_x = \left\{p \leq x\text{ prime }: \begin{array}{c} b\bmod{p} \in \langle a \bmod{p} \rangle \subset \BF_p^\ast \\ \text{or } \langle b \bmod{p} \rangle = \BF_p^\ast \end{array} \right\}.
\end{align*}

Let $p\in T_x$, and let $f_p(a)$ and $f_p(b)$ denote the order of $a$ and $b$ respectively in $\BF_p^\times$. Since $a$ and $b$ are not squares modulo $p$, it follows that 2 divides $f_p(a)$ and $f_p(b)$. From the definition of $T_x$, either $p-1 = 2q_1$ or $p-1 = 2q_1q_2$ with $q_1, q_2$ primes.

\emph{Case 1} Suppose $p-1 = 2q_1$. Since $f_p(a) \neq 2$, then $f_p(a) = 2q_1$ and so $a$ is a primitive root for $\BF^\times_p$. $p$ is therefore trivially in $S_x^\prime$.

\emph{Case 2} Suppose $p-1 = 2q_1q_2$, with $q_1<q_2$. There are three possibilities.

\emph{Case 2.1} $f_p(a) = 2q_1q_2$. Then, $a$ is a primitive root modulo $p$.

\emph{Case 2.2} $f_p(a) = 2q_2$.

\emph{Case 2.3} $f_p(a) = 2q_1$. We now show that this case does not happen too often. Here, clearly, $q_1 < \sqrt{x}$. We then count the number of $p\in T_x$ that can yield this situation. We do so by splitting the range of the possible $q_1$.

\emph{Case 2.3a} Suppose that $q_1 < \frac{\sqrt{x}}{\log x}$. Then if $p$ yields this case, $p$ divides $a^{2q_1} -1$ and the number of such primes when ranging over possible $q_1$ is at most
\begin{align*}
\sum_{q_1 < \sqrt{x}/\log x} \frac{2q_1}{\log x} \ll \frac{x}{(\log x)^3}
\end{align*}

\emph{Case 2.3b} Suppose that $ \frac{\sqrt{x}}{\log x} \leq q_1 < \sqrt{x}$. If $p$ yields this case, then we know that $\frac{p-1}{2q_1}$ has no small prime factor (in particular is equal to $q_2$). By a theorem of Bombieri, Friedlander and Iwaniec \cite{BoFrIw}, we know that for fixed $q_1< \sqrt{x}$,
$$\#\left\{ p\leq x \text{ prime } : \frac{p-1}{2q_1} \text{ has no small prime factors}\right\} \ll \frac{x}{q_1 (\log x)^2}.$$
Thus, summing over all possible $q_1$ in the range, we get that the number of $p$ that can yield this case is at most

$$\frac{x}{(\log x)^2} \sum_{\frac{\sqrt{x}}{\log x} \leq q_1 < \sqrt{x}} \frac{1}{q_1}.$$
Since we know that $\sum_{p<x} \frac{1}{p} = \log \log x + c +O\left(\frac{1}{\log x}\right)$, we get
\begin{align*}
\sum_{\frac{\sqrt{x}}{\log x} \leq q_1 < \sqrt{x}} \frac{1}{q_1} &= \log\log \sqrt{x} - \log\log\frac{\sqrt{x}}{\log x} + O\left(\frac{1}{\log x}\right)\\
&= \log \left(\frac{\frac{1}{2} \log x}{\frac{1}{2} \log x - \log\log x}\right) O\left(\frac{1}{\log x}\right)\\
&= -\log \left(1-\frac{2\log\log x}{\log x}\right) + O\left(\frac{1}{\log x}\right).
\intertext{For $x$ large enough, $\frac{2\log\log x}{\log x}$ is small, and for small $y$, $-\log(1-y) \sim y$. We then get}
\sum_{\frac{\sqrt{x}}{\log x} \leq q_1 < \sqrt{x}} \frac{1}{q_1} &\ll \frac{\log\log x}{\log x}.
\end{align*}
Therefore
$$\frac{x}{(\log x)^2} \sum_{\frac{\sqrt{x}}{\log x} \leq q_1 < \sqrt{x}} \frac{1}{q_1}\ll \frac{x \log\log x}{(\log x)^3}.$$

From the bounds we get in cases 2.3a and 2.3b, we conclude that the number of primes $p$ in $T_x$ yielding the case 2.3 is negligible compared to the total number of primes in $T_x$, which is at least $\frac{c x}{(\log x)^2}$. We thus have that
$$\left|\left\{ p\in T_x : a \text{ is a primitive root} \bmod{p} \text{ or } f_p(a) = 2q_2 \right\}\right| \gg \frac{x}{(\log x)^2}.$$
We can repeat the whole argument for $b$ instead of $a$ with $T_x$ replaced with the set above. We then get
$$\left|\left\{ p\leq x \text{ prime }: \begin{array}{c} a \text{ is a primitive root} \bmod{p} \text{ or } f_p(a) = 2q_2 \text{ and}\\b \text{ is a primitive root} \bmod{p} \text{ or } f_p(b) = 2q_2\end{array} \right\} \right| \gg \frac{x}{(\log x)^2}.$$

Now, if either $a$ or $b$ is a primitive root modulo $p$, then $p\in S^\prime_x$. Also, if $f_p(a) = f_p(b) = 2q_2$, then $\langle b \rangle = \langle a \rangle$ and so $p\in S^\prime_x$ as well.

We thus conclude that $\left|S^\prime_x\right| \gg \frac{x}{(\log x)^2}$ as desired.
\end{proof}

\section{Concluding remarks}
The original Artin conjecture was proved conditionally on the generalized Riemann hypothesis by Hooley (\cite{Ho}). The two-variable Artin conjecture was also proved conditionally on the generalized Riemann hypothesis by Moree and Stevenhagen (\cite{MoSt}). However, Theorem \ref{thmdis} suggests that we might not need the generalized Riemann hypothesis to show that at least one of them is true.

\section{Acknowledgments}
We would like to thank Professors Pieter Moree, Damien Roy and Peter Stevenhagen for helpful comments on a previous version of this paper.

\end{document}